\documentstyle[twoside, amsmath, amssymb, amsfonts]{article}

\newtheorem{theorem}{Theorem}[section]
\newtheorem{lemma}[theorem]{Lemma}
\newtheorem{remark}[theorem]{Remark}
\newtheorem{proposition}[theorem]{Proposition}
\newtheorem{corollary}[theorem]{Corollary}
\newtheorem{definition}[theorem]{Definition}
\newtheorem{example}[theorem]{Example}
\newenvironment{proof}{\begin{trivlist} \item[]{\em Proof.}}{\end{trivlist}}
\newcount\refno
\newcommand\be{\begin{equation}}
\newcommand\ee{\end{equation}}
\newcommand\bn{\begin{eqnarray}}
\newcommand\en{\end{eqnarray}}
\newcommand\bns{\begin{eqnarray*}}
\newcommand\ens{\end{eqnarray*}}
\newcommand\bd{\begin{definition}}
\newcommand\ed{\end{definition}}
\newcommand\br{\begin{remark}}
\newcommand\er{\end{remark}}
\newcommand\bt{\begin{theorem}}
\newcommand\et{\end{theorem}}
\newcommand\bp{\begin{proposition}}
\newcommand\ep{\end{proposition}}
\newcommand\bc{\begin{corollary}}
\newcommand\ec{\end{corollary}}
\newcommand\bl{\begin{lemma}}
\newcommand\el{\end{lemma}}
\newcommand\pf{\begin{proof}}
\newcommand\qed{\end{proof}\eop}

\newcommand\bR{{\mathbb R}}

\def\eop{\hfill\rule{2.0mm}{2.0mm}}

\makeindex
\begin{document}

\title{The $m$th-order Eulerian Numbers}
\author{Tian-Xiao He\\
{\small Department of Mathematics}\\
 {\small Illinois Wesleyan University}\\
 {\small Bloomington, IL 61702-2900, USA}\\
}
\date{}

\maketitle

\begin{abstract}
\noindent 

We define the $m$th-order Eulerian numbers with a combinatorial interpretation. The recurrence relation of the $m$th-order Eulerian numbers, the row generating function and the row sums of the $m$th-order Eulerian triangle are presented. We also define the $m$th-order Eulerian fraction and its alternative form. Some properties of the $m$th-order Eulerian fractions are represented by using differentiation and integration. An inversion relationship between second-order Eulerian numbers and Stirling numbers of the second kind is given. Finally, we give the exact expression of the values of the $m$th-order Eulerian numbers.

\vskip .2in \noindent AMS Subject Classification: 05A15, 05A05, 15B36, 15A06, 05A19, 11B73, 11B83.

\vskip .2in \noindent \textbf{Key Words and Phrases:} Eulerian numbers, second-order Eulerian numbers, $m$th-order Eulerian numbers, Eulerian fraction, $m$th-order Eulerian fraction, Stirling numbers.  

\end{abstract}

\pagestyle{myheadings} 
\markboth{T. X. He}
{$m$-order Eulerian Numbers}
\section{Instruction}

The (first order) Eulerian numbers represented by the values of the Eulerian triangle derived from Euler \cite[\S 13]{Euler} and \cite[P. 485]{Euler-2}. Their definition and properties can be found, for example, in the standard treatises in \cite{Aig, Com, GKP, Pet, Rio}, etc. However, their notations are not standard, which are shown by $A_{n,k}$, $A(n, k)$, $E(n,k)$, $\langle \begin{array}{cc}n\\k\end{array}\rangle$, etc. In this paper, we use the notation $T^{(1)}_{n,k}$ for the unification of $m$th-order Eulerian numbers, $m=1,2,\ldots$. The Eulerian number $T^{(1)}_{n,k}$ is also considered as the $(n,k)$th entry of the lower triangle matrix $(T^{(1)}_{n,k})_{ 0\leq k\leq n-1}$ that gives the number of permutations $\pi_1 \pi_2\ldots \pi_n$ in $[n]=\{1,2,\ldots, n\}$ that have exactly $k$ descents (or ascents), namely, $k$ places where $\pi_i>\pi_{i+1}$ (or $\pi_i<\pi_{i+1}$).  It is well known the Eulerian numbers satisfy the symmetric law 

\[
T^{(1)}_{n,k}=T^{(1)}_{n, n-k-1}
\]
and the recurrence relation 

\[
T^{(1)}_{n,k}=(k+1)T^{(1)}_{n-1,k}+(n-k)T^{(1)}_{n-1,k-1}, \quad 0\leq k \leq n-1,
\]
with $T^{(1)}_{1,0}=1$ and $T^{(1)}_{n,k}=0$ for $k<0$ and $k\geq n$. 

Eulerian numbers are useful primarily because they provide an unusual connection between ordinary powers,  consecutive binomial coefficients, Stirling numbers of the first kind and the second kind, Bernoulli numbers, as well as $B$-spline functions. Some related materials and literatures can be found in \cite{Com, GKP, Pet, Rio, He12} and their references. Recently, Butzer, Markett and the author \cite{BHM} study Eulerian fractions and Stirling, Bernoulli and Euler functions with complex order parameters and their impact on the polylogarithm functions. 

The entries of a similar triangle denoted by $(T^{(2)}_{n,k})_{0\leq k\leq n-1}$ are called second-order Eulerian numbers because of the following recurrence relation similar to the recurrence relation of $T^{(1)}_{n,k}$:

\[
T^{(2)}_{n,k}=(k+1)T^{(2)}_{n-1,k}+(2n-k-1)T^{(2)}_{n-1,k-1}, \quad 0\leq k\leq n-1,
\]
with $T^{(2)}_{1,0}=1$ and $T^{(1)}_{n,k}=0$ for $k<0$ and $k\geq n$. 

The numbers $T^{(2)}_{n,k}$ are investigated by Carlitz \cite{Car1,Car2} and Riordan \cite[Sect. 4]{Rio76} using a different notation. Carlitz derives their formal properties and asks for a combinatorial interpretation.
Riordan states such an interpretation, but not in terms of descents of permutations as that is represented in 
Gessel and Stanley \cite{GS}. If we form permutations of the multiset $[1,1,2,2,\ldots, n,n]$ with the special property that all numbers between the two occurrences of $u$ are less than $u$, for $1\leq u\leq n$, then
­­$T^{(2)}_{n,k}$ is the number of such permutations that have $k$ descents. For example, there are eight suitable single-descent permutations of $[1,1,2,2,3,3]$:

\[
113322,113223,221133, 211233, 223311,223113,331122, 311223.
\]

Second-order Eulerian numbers are important chiefly because of their connection with Stirling numbers shown in Ginsburg \cite{Gis}. Recently, there are several interesting results related second-order Eulerian numbers interested combinatorial people. For instance, Rz\c{a}dkowski and Urli\'anska \cite{RU} find  the connection between second-order Eulerian numbers and Bernoulli numbers, which was precisely proved in Fu\cite{Fu}. O'Sullivan \cite{Osu} use second-order Eulerian numbers, Bernoulli numbers, Stirling cycle numbers to give a combinatorial proof of a remarkable conjecture relevant two sequences introduced by Ramanujan. 

In next section, we extend Eulerian numbers and second-order Eulerian numbers to $m$th-order Eulerian numbers with a combinatorial interpretation. The recurrence relation of the $m$th-order Eulerian numbers, the row generating function and the row sums of the $m$th-order Eulerian triangle are given. In Section3, we define the $m$th-order Eulerian fraction and its alternative form. Some properties of the $m$th-order Eulerian fractions are represented by using differentiation and integratio. A formula for evaluating Stirling numbers of the second kind in terms of second-order Eulerian 
numbers is presented and extended to $m$th-order setting. Naturally, the inverse process to find second-order Eulerian numbers by using Stirling numbers of the second kind is also considered. In the final section, we give the exact expression of the values of the $m$th-order Eulerian numbers.

\section{The $m$th-order Eulerian Numbers}

We now define $m$th-order Eulerian numbers.

\begin{definition}\label{def:2.1}
Let $Q_{mn}$ be the set of all permutations $a_1a_2 \ldots a_{mn}$ of the multisef $M_n=\{1,\ldots,1,2,\ldots, 2,\ldots, n,\ldots, n\}$, with the multiplicity of $m$ for each integer, such that if $u < v < w$ and $a_u = a_w$, then $a_u\geq a_v$. (We call the elements of $Q_{mn}$ $m$-Stirling permutations.) We denote by $T^{(m)}_{n,k}$ the number of permutations $a_1a_2\ldots a_{mn}\in Q_{mn}$ with exactly $k$ descents, i,e., such that $a_j > a_{j+1}$ or $j = mn$ for exactly $k$ values of $j\in [mn]=\{ 1, 2,\ldots, mn\}$. Let $T^{(m)}_{n,k}$ be the number of permutations in $Q_{mn}$ with $k$ descents. Thus $T^{(m)}_{1,1}=1$, $T^{(m)}_{1,k}=0$ for $k\not= 1$, and $T^{(m)}(n,k)=0$ for all $k< 0$ and $k\geq n$. 

$T^{(m)}_{n,k}$, the entries of the triangle $(T^{(m)}_{n,k})_{0\leq k \leq n-1}$, are called the $m$th-order Eulerian numbers. If $m=1, 2$ and $3$, $T^{(1)}_{n,k}$, $T^{(2)}_{n,k}$, and $T^{(3)}_{n,k}$ are called (first-order) Eulerian numbers, second-order and third-order Eulerian numbers, respectively. 
\end{definition}

For instance, for $m=2$ and $n = 3$ there are $15$ such permutations, $1$ with no descents, $8$ with a single descent, and $6$ with two descents:

\begin{align*}
&112233;\\
&113322,113223,221133, 211233, 223311,223113,331122, 311223;\\
&332211,322311,311322,332112,322113,321123.\\
\end{align*}
Hence, $T^{(2)}_{3,0}=1$, $T^{(2)}_{3,1}=8$, and $T^{(2)}_{3,2}=6$. 

\begin{theorem}\label{thm:2.2}
Let $T^{(m)}_{n,k}$ be the $m$th-order Eulerian numbers defined by Definition \ref{def:2.1}. Then for $m\geq 1$ and $0\leq k\leq n-1$ we have the recurrence relation of $T^{(m)}_{n,k}$,  

\be\label{2.1}
T^{(m)}_{n,k}=(k+1)T^{(m)}_{n-1,k}+(mn-k-m+1)T^{(m)}_{n-1,k-1},
\ee
where $T^{(m)}_{1,0}=1$, $T^{(m)}_{1,k}=0$ for $k\not= 0$, and $T^{(m)}_{n,k}=0$ for $k<0$ or $k\geq n$.

Particularly, for $m=1, 2,$ and $3$ we have, respectively, the recurrence relations of Eulerian numbers, second-order Eulerian numbers, and the third Eulerian numbers,

\begin{align}
&T^{(1)}_{n,k}=(k+1)T^{(1)}_{n-1,k}+(n-k)T^{(1)}_{n-1,k-1},\label{2.2}\\
&T^{(2)}_{n,k}=(k+1)T^{(2)}_{n-1,k}+(2n-k-1)T^{(2)}_{n-1,k-1},\label{2.3}\\
&T^{(3)}_{n,k}=(k+1)T^{(3)}_{n-1,k}+(3n-k-2)T^{(3)}_{n-1,k-1},\label{2.4}
\end{align}
where $T^{(1)}_{n,k}$ and $T^{(2)}_{n,k}$ are also denoted by $\left\langle \begin{array}{c}n\\k\end{array}\right\rangle$ and 
$\left\langle\left\langle \begin{array}{c}n\\k\end{array}\right\rangle\right\rangle$, respectively, in some literatures. 
\end{theorem}

\begin{proof}
We can use the descent definition of $m$-order Eulerian numbers to prove the linear recurrence relation as follows. If $w$ is in $Q_{mn}$ with $k$ descents, then deleting element $m$$n$ from $w$ results in a permutation in $Q_{m(n-1)}$ with $k$ or $k-1$ descents. Conversely, we can form permutation in $Q_{mn}$ with $k$ descents from permutations in $Q_{m(n-1)}$ with $k$ or $k-1$ descents by inserting $m$$n$. 

More precisely, suppose $v$ is a permutation in $Q_{m(n-1)}$ with $k-1$ descents. The inserting $m$$n$ at the far left of $v$ or in an descent position of $v$ creates a permutation $w\in Q_{mn}$ with $k$ descents. There are $m(n-1)-(k-1)=mn-k-m+1$ such positions. 

Similarly, if $v\in Q_{m(n-1)}$ already has $k$ descents, then inserting $m$$n$ in a descent position of $v$ or at the far right gives a permutation $w\in Q_{mn}$ with $k$ descents. There are $k+1$ such positions. Hence, we have established the recurrence relations \eqref{2.1}-\eqref{2.4}
\end{proof}

The second-order Eulerian numbers satisfy the recurrence relation \eqref{2.3}, that follows directly from the above definition. For instance, from $T^{(2)}_{2,0}=0$ and $T^{(2)}_{2,1}=1$, 

\begin{align*}
&T^{(2)}_{3,1}=1\cdot T^{(2)}_{2,1}+(2\cdot 3-1)\cdot T^{(2)}_{2,0}=1,\\
&T^{(2)}_{3,2}=2\cdot T^{(2)}_{2,2}+(2\cdot 3-2)\cdot T^{(2)}_{2,1}=2\cdot 2+4\cdot 1=8,\\
&T^{(2)}_{3,3}=3\cdot T^{(2)}_{2,3}+(2\cdot 3-3)\cdot T^{(2)}_{2,2}=0+3\cdot 2=6.
\end{align*}

The following table gives the first few values of $T^{(m)}_{n,k}$ for $0\leq k\leq n-1$.

{\footnotesize  
\begin{align}\label{2.5}
&\left[T^{(m)}_{n,k}\right]=\nonumber\\
&\left[\begin{array}{llllll}
1 &  &  &  &  &  \\ 
1 & m&  &  &  &  \\ 
1 & 4m& m(2m-1)&  &  &   \\ 
1 & 11m & m(18m-7)& m(6m^2-7m+2) &  &  \\ 
1 & 26m & 2m(49m-16)& 2m(48m^2-46m+11) & m(4m-3)(6m^2-7m+2) &  \\ 
   &  & \cdots  &  &  & 
\end{array}
\right].
\end{align}
}
Therefore, for $m=1,2,$ and $3$, we have 

\be\label{2.6}
\left[T^{(1)}_{n,k}\right]_{0\leq k \leq n-1}=\left[\begin{array}{ccccccc}
1 &  &  &  &  &  &\\ 
1 & 1 &  &  &  &  &\\ 
1 & 4& 1 &  &  & &\cdots  \\ 
1 & 11 & 11& 1 &  & & \\ 
1 & 26 & 66& 26 & 1 &&  \\ 
   &  & \cdots  &  &  & &
\end{array}
\right].
\ee

\be\label{2.7}
\left[T^{(2)}_{n,k}\right]_{0\leq k \leq n-1}=\left[\begin{array}{ccccccc}
1 &  &  &  &  &  &\\ 
1 & 2 &  &  &  &  &\\ 
1 & 8& 6 &  &  & &\cdots  \\ 
1 & 22 & 58& 24 &  & & \\ 
1 & 52 & 328& 444 & 120 &&  \\ 
   &  & \cdots  &  &  & &
\end{array}
\right].
\ee

\be\label{2.8}
\left[T^{(3)}_{n,k}\right]_{0\leq k \leq n-1}=\left[\begin{array}{ccccccc}
1 &  &  &  &  &  &\\ 
1 & 3 &  &  &  &  &\\ 
1 & 12& 15 &  &  & &\cdots  \\ 
1 & 33 & 141& 105 &  & & \\ 
1 & 78 & 786& 1830 & 945 &&  \\ 
   &  & \cdots  &  &  & &
\end{array}
\right].
\ee

\be\label{2.8-2}
\left[T^{(4)}_{n,k}\right]_{0\leq k \leq n-1}=\left[\begin{array}{ccccccc}
1 &  &  &  &  &  &\\ 
1 & 4 &  &  &  &  &\\ 
1 & 16& 28 &  &  & &\cdots  \\ 
1 & 44 & 260& 280 &  & & \\ 
1 & 104 & 1440& 4760 & 3640 &&  \\ 
   &  & \cdots  &  &  & &
\end{array}
\right].
\ee

A $3$-Stirling permutation of order $n$ is a permutation of the multiset $\{1,1,1,2,2,2,...,n,n,n\}$ such that for each $i$, $1\leq i\leq n$, the elements occurring between two occurrences of $i$ are at most $i$. The triangle enumerates $3$-Stirling permutations by descents form third-order Eulerian numbers. Clearly, the entries of the triangle satisfy the recurrence equation, $T^{(3)}_{n+1,k}= (k + 1)T^{(3)}_{n,k} + (3n-k-2)T^{(3)}_{n,k-1}$, which is the sequence $A219512$ shown in OEIS \cite{Sloane}

The $4$-Stirling permutation of order $n$ is a permutation of the multiset $\{1,1,1,1,2,2,2,2,...,n,n,n,n\}$ such that for each $i$, $1\leq i\leq n$, the elements occurring between two occurrences of $i$ are at most $i$. The triangle enumerated by $4$-Stirling permutation of order $n$ by descents, i.e., $(T^{(4)}_{n,k})_{0\leq k\leq n-1}$ with $k$ descents, form 
fourth-order Eulerian numbers. The triangle possesses the recurrence equation, $T^{(4)}_{n+1,k}= (k + 1)T^{(4)}_{n,k} + (4n-k-3)T^{(3)}_{n,k-1}$, which first few entries are $\{1,1,4,1,16,28,1,44,260,$ $280, 1,104,1440,4760,3640,\ldots\}$. The fourth order Eulerian number sequence is unknown in \cite{Sloane}. 

\begin{theorem}\label{thm:2.3}
Let $T^{(m)}_{n,k}$ be the $m$th-order Eulerian numbers defined in Definition \ref{def:2.1}. Denote by 

\be\label{2.9}
S_{m;n}(t)=\sum^{n-1}_{k=0}T^{(m)}_{n,k} t^k,
\ee
the $m$th-order Eulerian polynomials, where $S_{1;n}(t)$ are the classical Eulerian polynomials. Then 

\be\label{2.10}
S_{m;n+1}(t)=(1+mn t)S_{m;n}(t)+t(1-t)S'_{m;n}(t).
\ee
\end{theorem}

\begin{proof}
From \eqref{2.9}, we have 

\[
S'_{m;k}(t)=\sum^{n-1}_{k=0}kT^{(m)}_{n,k} t^{k-1}, 
\]

\[
t(1-t)S'_{m;k}(t)=\sum^{n-1}_{k=0}kT^{(m)}_{n,k} t^{k}-\sum^{n}_{k=1}(k-1)T^{(m)}_{n,k-1} t^{k},
\]
and 

\begin{align*}
&(1+mnt)S_{m;n}(t)=(1+mnt)\sum^{n-1}_{k=0}T^{(m)}_{n,k}t^k\\
=& \sum^{n-1}_{k=0}T^{(m)}_{n,k}t^k+\sum^{n}_{k=1}mnT^{(m)}_{n,k-1}t^k.
\end{align*}
Hence, noting \eqref{2.1}, we have 

\begin{align*}
&(1+mnt)S_{m;n}(t)+t(1-t)S'_{m;n}(t)\\
=&\sum^{n-1}_{k=0}(k+1)T^{(m)}_{n,k} t^{k}
+\sum^{n}_{k=1}(mn-k+1)T^{(m)}_{n,k-1} t^{k}\\
=&\sum^{n}_{k=0}\left((k+1)T^{(m)}_{n,k} +(mn-k+1)T^{(m)}_{n,k-1}\right)t^{k}\\
=&\sum^{n}_{k=0}T^{(m)}_{n+1,k} t^{k},
\end{align*}
which yields \eqref{2.10}.
\end{proof}

\begin{corollary}\label{cor:2.4}
Let $T^{(m)}_{n,k}$ be the $m$th-order Eulerian numbers defined in Definition \ref{def:2.1}. Then

\begin{align}\label{2.11}
&\sum^n_{k=0}T^{(m)}_{n+1,k}=(1+nm)\sum^{n-1}_{k=0}T^{(m)}_{n,k},\\
&\sum^{n-1}_{k=0} T^{(m)}_{n,k}=1\cdot (m+1)\cdot (2m+1)\cdots ((n-1)m+1).\label{2.12}
\end{align}
\end{corollary}

\begin{proof}
Substituting $t=1$ into \eqref{2.10} yields

\be\label{2.11-1}
S_{m;n+1}(1)=(1+mn)S_{m;n}(1),
\ee
i.e.

\[
\sum^n_{k=0}T^{(m)}_{n+1,k}=(1+nm)\sum^{n-1}_{k=0}T^{(m)}_{n,k}.
\]
From the last expression and noting $T^{(m)}_{1,0}=1$, we obtain \eqref{2.12}.
\end{proof}

The $m$-th order Eulerian polynomials and their $q$-forms will be shown in subsequent articles. 

\section{$m$th-order Eulerian fraction}

We now define $m$th-order Eulerian fraction and its useful alternative form.

\begin{definition}\label{def:3.1}
Let $S_{m;n}(t)$ be the $m$th-order Eulerian polynomials defined by \eqref{2.9}. We define the $m$th-order Eulerian fraction by 

\be\label{3.1}
F_{m;n}(t)=\frac{S_{m;n}(t)}{(1- t)^{m(n-1)+2}},
\ee
and an alternative form of \eqref{3.1} is defined by 

\be\label{3.1-2}
\widehat F_{m;n}(t)=\frac{tS_{m;n}(t)}{(1- t)^{mn+1}}.
\ee
Particularly, the corresponding first-order Eulerian fraction and its alternative form  

\begin{align*}
&F_{1;n}(t)=\frac{S_{1;n}(t)}{(1-t)^{n+1}}\quad and\\
&\widehat F_{1;n}(t)=\frac{tS_{1;n}(t)}{(1-t)^{n+1}}
\end{align*}
are the (classical) Eulerian fractions (cf. for example, \cite{Com, HHST}).
\end{definition}

Here, $F_{2;n}(t)$ is the second-order extension of the Eulerain fraction defined in \cite{Osu}, while $\widehat F_{2;n}(t)$ is the second-order extension of the Eulerian fraction defined in \cite{GS}.

We now give the relationship between two forms of Eulerian fractions and consider some properties of those $m$th-order Eulerian fractions represented by using differentiation and integration.
 
\begin{theorem}\label{thm:3.2}
Let $S_{m;n}(t)$ be the $m$th-order Eulerian polynomials defined by \eqref{2.9}, and let $F_{m;n}(t)$ and $\widehat F_{m;n}(t)$ be the $m$th-order Eulerian fraction and its alternative form defined by \eqref{3.1} and \eqref{3.1-2}, respectively. Then 

\begin{align} 
&\widehat F_{m;n}(t)=\frac{t}{(1-t)^{m-1}}F_{m;n}(t),\label{3.2}\\
&\frac{d}{dt} \left( \frac{t}{(1- t)^{m-1}}F_{m;n}(t)\right)=F_{m;n+1}(t),\label{3.2-2}\\
&t\frac{d}{dt}\widehat F_{m;n}(t)=(1-t)^{m-1}\widehat F_{m;n+1}(t).\label{3.2-3}
\end{align}
Particularly, for the first-order Eulerian fraction 

\begin{align*}
&\widehat F_{1;n}(t)=tF_{1;n}(t),\\
&\frac{d}{dt}(tF_{1;n}(t))=F_{1;n+1}(t),\\
&t\frac{t}{dt}\widehat F_{1;n}(t)=\widehat F_{1;n+1}(t).
\end{align*}
\end{theorem}

\begin{proof}
Equation \eqref{3.2} is obvious. 

To prove \eqref{3.2-2}, we divid both sides of \eqref{2.10} by $(1- t)^{mn+2}$ to get  

\begin{align*}
&F_{m;n+1}(t)=\frac{S_{m;n+1}(t)}{(1- t)^{mn+2}}\\
=&\frac{1+mnt}{(1-t)^{mn+2}}S_{m;n}(t)+\frac{t}{(1- t)^{mn+1}}S'_{m;n}(t)\\
=& \frac{d}{dt} \left(\frac{t}{(1- t)^{mn+1}}\right)S_{m;n}(t)+\frac{t}{(1- t)^{mn+1}}S'_{m;n}(t)\\
=&\frac{d}{dt} \left( \frac{t}{(1- t)^{m-1}}\frac{S_{m;n}(t)}{(1- t)^{m(n-1)+2}}\right),
\end{align*}
which implies \eqref{3.2} by noting \eqref{3.1}.

From \eqref{3.2} we have 

\begin{align*}
&t\frac{d}{dt}\widehat F_{m;n}(t)=t\frac{d}{dt} \left( \frac{t}{(1- t)^{m-1}}F_{m;n}(t)\right)\\
=&tF_{m;n+1}(t)=t\frac{(1-t)^{m-1}}{t}\widehat F_{m;n+1}(t),
\end{align*}
which yields \eqref{3.2-3}.
\end{proof}

For $m=2$, \eqref{3.2-2} reduces to a result in \cite{Osu}:

\[
\left( \frac{t}{1-t}F_{2;n}(t)\right)'=F_{2;n+1}(t).
\]

For $m=2$, \eqref{3.2-3} becomes to 

\[
t\frac{t}{dt}\widehat F_{2;n}(t)=\widehat F_{2;n+1}(t),
\]
which is shown in \cite{GS}. 

\begin{proposition}\label{pro:3.3}
Let $F_{m;n}(t)$ and $\widehat F_{m;n}(t)$ be the two forms of Eulerian fractions defined by \eqref{3.1} and \eqref{3.1-2}, respectively, and let $m$, $n$, $a$, and $b$ be the integers with $m, n\geq 1$, $a\geq 0$, and $n+a+b\geq 0$. Then 

\begin{align}
&\int^0_{-\infty}t^a(1- t)^{-a-b-2}F_{m;n}(t)dt\nonumber\\
=&\frac{(-1)^a}{mn+a+b+1}\sum^{n-1}_{k=0}
(-1)^kT^{(m)}_{n,k}\binom{mn+a+b}{a+k}^{-1}, \label{3.3}\\
&\int^0_{-\infty}t^{a-1}(1- t)^{m-a-b-3}\widehat F_{m;n}(t)dt\nonumber\\
=&\frac{(-1)^a}{mn+a+b+1}\sum^{n-1}_{k=0}
(-1)^kT^{(m)}_{n,k}\binom{mn+a+b}{a+k}^{-1}. \label{3.3-2}
\end{align}
\end{proposition}

\begin{proof}
Noticing that $S_{m;n}(t)$ and $F_{m;n}(t)$ are the functions defined by \eqref{2.9} and \eqref{3.1}, respectively, we may calculating the integral 

\begin{align*}
&\int^{1}_0t^a(1- t)^bF_{m;n}\left( \frac{t}{t-1}\right)dt\\
=&\sum^{n-1}_{k=0}T^{(m)}_{n,k} \int^{1}_{0} t^a (1- t)^b \frac{\left(\frac{t}{ t-1}\right)^k}{(1- (t/( t-1)))^{mn}}\\
=&\sum^{n-1}_{k=0}(-1)^kT^{(m)}_{n,k} \int^{1}_{0} t^{a+k} (1- t)^{b-k+mn}dt\\
=& \sum^{n-1}_{k=0}
(-1)^kT^{(m)}_{n,k}\frac{\Gamma(a+k+1)\Gamma (b-k+mn+1)}{\Gamma (a+b+mn+2)},
\end{align*}
which gives a $(-1)^a$ multiple of the right-hand side of \eqref{3.3}. Using the transformation $x=t/(t-1)$ into the leftmost side integral of the last equations, we have 

\begin{align*}
&\int^{1}_0t^a(1- t)^bF_{m;n}\left( \frac{t}{ t-1}\right)dt\\
=&\int^{-\infty}_{0}\left(\frac{x}{ x-1}\right)^a\left( 1-\frac{ x}{ x-1}\right)^bF_{m;n}(x)d\frac{x}{x-1}\\
=&(-1)^a\int^0_{-\infty} x^a(1- x)^{-a-b-2}F_{m;n}(x)dx,
\end{align*}
which is a $(-1)^a$ multiple of the left-hand side of \eqref{3.3}. 

Equation \eqref{3.3-2} follows from \eqref{3.3} and \eqref{3.2}, which completes the proof of the proposition. 
\end{proof}

\begin{theorem}\label{thm:3.4}
Let $\widehat F_{m;n}(t)$ be the Eulerian fraction defined by \eqref{3.1-2}. Denote $\widehat F_{m;n}(t)=\sum_{\ell \geq 0} f_{m;n}(\ell)t^\ell$. Then 

\be\label{3.4}
f_{m;n}(\ell)=\sum^{\ell-1}_{k=0}T^{(m)}_{n,k}\binom{mn+\ell-k-1}{mn},
\ee
where $T^{(m)}_{n,k}$ are the $m$th-order Eulerian numbers defined by Definition \ref{def:2.1}. 

Particularly, for $m=2$ we have 

\be\label{3.4-2}
f_{2;n}(\ell)=\sum^{\ell-1}_{k=0}T^{(2)}_{n,k}\binom{2n+\ell-k-1}{2n},
\ee
which implies 

\be\label{3.4-3}
S(n+\ell, \ell)=\sum^{\ell-1}_{k=0}T^{(2)}_{n,k}\binom{2n+\ell-k-1}{2n},
\ee
where $S(u,v)$ is the $(u,v)$th Stirling number of the second kind. Since $S(u,v)=S(u-v+v,v)$, \eqref{3.4-3} gives 
a scheme for computing $S(u,v)$ by substituting in $(n,\ell)=(u-v, v)$. 
\end{theorem}

\begin{proof}
From \eqref{3.1-2}, 

\begin{align*}
\widehat F_{m;n}(t)=&(1-x)^{-mn-1}\sum^n_{k=1}T^{(m)}_{n,k-1}x^k\\
=&\sum^n_{k=1} \sum_{j\geq 0}T^{(m)}_{n,k-1}\binom{mn+j}{j}x^{k+j}\\
=&\sum^n_{k=1} \sum_{\ell \geq k}T^{(m)}_{n,k-1}\binom{mn+\ell-k}{mn}x^{\ell}\\
=&\sum_{\ell \geq 1}\sum^\ell_{k=1}T^{(m)}_{n,k-1}\binom{mn+\ell-k}{mn}x^{\ell}\\
=&\sum_{\ell \geq 1}\sum^{\ell-1}_{k=0}T^{(m)}_{n,}\binom{mn+\ell-k-1}{mn}x^{\ell}.
\end{align*}
Denote $\widehat F_{m;n}(t)=\sum_{\ell \geq 0} f_{m;n}(\ell)t^\ell$. By comparing two expressions of $\widehat F_{m;n}$, we may obtain \eqref{3.4}. 

If $m=2$, then \eqref{3.4} reduces to \eqref{3.4-2}. From \cite{GS}, we have 

\[
\widehat F_{2;n}(t)=\sum_{\ell \geq 0} f_{2;n}(\ell)t^\ell=\sum_{\ell \geq 0} S(\ell+n,\ell)t^\ell,
\]
where $S(n,k)$ is a Stirling number of the second kind, which derives \eqref{3.4-3}.
\end{proof}

\begin{example}
For examples, to find $S(3,1)=S(2+1,1)$, $S(5,2)=S(3+2,2)$ from \eqref{3.4-3}, we set in $\widehat F_{2;n}(\ell)$ $(n,\ell)=(2,1)$ and $(3,2)$, respectively, and have 

\begin{align*}
&S(3,1)=S(1+2,1)=\widehat F_{2;2}(1)=\sum^{1-1}_{k=0}T^{(2)}_{2,k}\binom{4+1-k-1}{4}\\
\qquad &=T^{(2)}_{2,0}\cdot 1=1,\\
&S(5,2)=S(3+2,2)=\widehat F_{2;3}(1)=\sum^{2-1}_{k=0}T^{(2)}_{2,k}\binom{6+2-k-1}{6}\\
\qquad &=T^{(2)}_{3,0}\cdot 7+T^{(2)}_{3,1}\cdot 1=1\cdot 7+8\cdot 1=15.
\end{align*}

Similarly, we have the first few values of $(f_{3,n}(\ell))_{1\leq \ell \leq n}$ for $n=1,2,$ and $3$:

\begin{align*}
f_{3,1}(1)=&\sum^{1-1}_{k=0}T^{(3)}_{1,k}\binom{3+1-k-1}{3}=T^{(3)}_{1,0}\cdot 1=1,\\
f_{3,2}(1)=&\sum^{1-1}_{k=0}T^{(3)}_{2,k}\binom{6+1-k-1}{6}=T^{(3)}_{2,0}\cdot 1=1,\\
f_{3,2}(2)=&\sum^{2-1}_{k=0}T^{(3)}_{2,k}\binom{6+2-k-1}{6}=T^{(3)}_{2,0}\cdot 7+T^{(3)}_{2,1}\cdot 1=1\cdot 7+3\cdot 1=10,\\
f_{3,3}(1)=&\sum^{1-1}_{k=0}T^{(3)}_{3,k}\binom{9+1-k-1}{9}=T^{(3)}_{3,0}\cdot  1=1,\\
f_{3,3}(2)=&\sum^{2-1}_{k=0}T^{(3)}_{3,k}\binom{9+2-k-1}{9}=T^{(3)}_{3,0}\cdot 10+T^{(3)}_{3,1}\cdot 1=1\cdot 10+12\cdot 1=22,\\
f_{3,3}(3)=&\sum^{3-1}_{k=0}T^{(3)}_{3,k}\binom{9+3-k-1}{9}=T^{(3)}_{3,0}\cdot 55+T^{(3)}_{3,1}\cdot 10+T^{(3)}_{3,2}\cdot 1\\
=&1\cdot 55+12\cdot 10+15\cdot 1=190.\\
\end{align*}

Hence,

\be\label{2.8}
\left[f_{3;n}(\ell)\right]_{1\leq \ell \leq n}=\left[\begin{array}{ccccccc}
1 &  &  &  &  &  &\\ 
1 & 10 &  &  &  &  &\\ 
1 & 22& 190 &  &  & &\cdots  \\ 
1 & 46 & 661& 5396 &  & & \\ 
1 & 94 & 2170& 25830 & 204645 &&  \\ 
   &  & \cdots  &  &  & &
\end{array}
\right].
\ee
The the generating function of the second column of the last triangle is $3 \cdot 2^n-2$ for $n=2,3,\ldots.$ However, the 
triangle itself seems new. 
\end{example}

The inverse of \eqref{3.4} can be presented as below.

\begin{theorem}\label{thm:3.5}
Let $T^{(m)}_{n,k}$ be the $m$th-order Eulerian numbers defined by Definition \ref{def:2.1}, and let $\widehat F_{m;n}(t)$ be the Eulerian fraction defined by \eqref{3.1-2}. Denote $\widehat F_{m;n}(t)=\sum_{\ell \geq 0} f_{m;n}(\ell)t^\ell$ with $f_{m;n}(0)=0$ for $n\geq 1$. Then 

\be\label{3.5}
T^{(m)}_{n,k}=\sum^{k+1}_{\ell=1}(-1)^{k-\ell+1}\binom{mn+1}{k-\ell+1}f_{m;n}(\ell)
\ee
for $n\geq 1$. Particularly, for $m=2$ and $n\geq 1$, 

\begin{align}\label{3.5-2}
&T^{(2)}_{n,k}=\sum^{k+1}_{\ell=1}(-1)^{k-\ell+1}\binom{2n+1}{k-\ell+1}f_{2;n}(\ell)\nonumber\\
=&\sum^{k+1}_{\ell=1}(-1)^{k-\ell+1}\binom{2n+1}{k-\ell+1}S(n+\ell, \ell),
\end{align}
where $S(n,k)$ is a Stirling number of the second kind. 
\end{theorem}

\begin{proof}
From \eqref{2.9} and \eqref{3.1-2}, 

\begin{align*}
tS_{m;n}(t)=&\sum^n_{k=1}T^{(m)}_{n,k-1}t^k\\
=& (1-t)^{mn+1}\sum_{\ell\geq 0}f_{m;n}(\ell)t^\ell\\
=&\sum_{j=0}^{mn+1}\sum_{\ell \geq 0}(-1)^{j}\binom{mn+1}{j}f_{m;n}(\ell)t^{j+\ell}\\
=&\sum_{k\geq 0}\left(\sum^k_{\ell=0}(-1)^{k-\ell}\binom{mn+1}{k-\ell}f_{m;n}(\ell)\right) t^k,
\end{align*}
which implies 

\[
T^{(m)}_{n,k-1}=\sum^k_{\ell=0}(-1)^{k-\ell}\binom{mn+1}{k-\ell}f_{m;n}(\ell)
\]
for $k\geq 1$. Thus, \eqref{3.5} follows. For $m=2$, by noting $f_{2;n}(\ell)=S(n+\ell, \ell)$, we obtain \eqref{3.5-2}.
\end{proof}

\begin{example}
For examples, to find second-order Eulerian numbers $T^{(2)}_{2,0}$, we substitute $(n,k)=(2,0)$ into \eqref{3.5-2} and  obtain  

\[
T^{(2)}_{2,0}=\sum^1_{\ell=1}(-1)^{1-\ell}\binom{5}{1-\ell}S(2+\ell,\ell)=S(3,1)=1.
\]
Similarly, 

\[
T^{(2)}_{3,1}=\sum^2_{\ell=1}(-1)^{2-\ell}\binom{7}{2-\ell}S(3+\ell,\ell)=-7\cdot 1+1\cdot 15=8.
\]
and 

\[
T^{(2)}_{4,2}=\sum^3_{\ell=1}(-1)^{3-\ell}\binom{9}{3-\ell}S(4+\ell, \ell)=36\cdot 1-9\cdot 31+1\cdot 301=58.
\]
\end{example}

Similarly, from (5) in \cite{GS}, we obtain the inverse relation between Stirling numbers of the first kind and the second-order Eulerian numbers.

\begin{theorem}\label{thm:3.6}
Let $T^{(m)}_{n,k}$ be the $m$th-order Eulerian numbers defined by Definition \ref{def:2.1}. Denote the signed Stirling numbers of the first kind and the unsigned Stirling numbers of the first kind by and $s(n,k)$ and $\genfrac{[}{]}{0pt}{}{n}{k}$, respectively. Then 

\begin{align}\label{3.5-3}
&\genfrac{[}{]}{0pt}{}{n}{n-k}=(-1)^{n-k}s(n,n-k)= \sum^n_{i=k+1}T^{(2)}_{k,2k-i}\binom{2k+n-i}{2k},\\
&T^{(2)}_{k, 2k-i}=\sum^i_{n=0}(-1)^{i-n}\genfrac{[}{]}{0pt}{}{n}{n-k}\binom{2k+1}{i-n}.\label{3.5-4}
\end{align}
\end{theorem}

\begin{proof}
From (5) in \cite{GS} with a few modification and noting $B_{k,2k-i+1}=T^{(2)}_{k,2k-i}$, we have 

\begin{align*}
&\sum_{n\geq 0} g_k(n)x^n=\sum_{n\geq 0} \genfrac{[}{]}{0pt}{}{n}{n-k}x^n\\
=& \sum^{2k}_{i=k+1}B_{k,2k-i+1}x^i\sum_{j\geq 0}\binom{2k+j}{j}x^j\\
=&\sum^{2k}_{i=k+1}\sum_{j\geq 0} T^{(2)}_{k,2k-i}\binom{2k+j}{j}x^{i+j}\\
=&\sum^{2k}_{i=k+1}\sum_{n\geq i} T^{(2)}_{k,2k-i}\binom{2k+n-i}{2k}x^{n}\\
=&\sum_{n\geq 0}\sum^n_{i=k+1}T^{(2)}_{k,2k-i}\binom{2k+n-i}{2k}x^{n}.
\end{align*}
By comparing the coefficients of the power $x^n$ on the leftmost side and rightmost side of the last expression, we may obtain \eqref{3.5-3}.

Also from (5) in \cite{GS} with modification, we have 

\begin{align*}
& \sum^{2k}_{i=k+1}B_{k, 2k-i+1}=\sum^{2k}_{i=k+1}T^{(2)}_{k,2k-i}\\
=&\sum_{n\geq 0}\sum_{j\geq 0} g_k(n)\binom{2k+1}{j}(-1)^jx^{n+j}\\
=&\sum_{n\geq 0}\sum_{i\geq n} g_k(n)\binom{2k+1}{i-n}(-1)^{i-n}x^{i}\\
=&\sum_{i\geq 0}\sum^i_{n= 0}(-1)^{i-n}\genfrac{[}{]}{0pt}{}{n}{n-k}\binom{2k+1}{i-n}x^i.
\end{align*}
By comparing the coefficients of the power $x^n$, we obtain \eqref{3.5-4}.
\end{proof}

\begin{example}
We may use \eqref{3.5-3} to evaluate $\genfrac{[}{]}{0pt}{}{u}{v}$ and $s(u,v)$ by substituting into $n=u$ and $k=u-v$. For instance, 
to find $\genfrac{[}{]}{0pt}{}{2}{1}$ by using \eqref{3.5-3}, we substitute $n=2$ and $k=1$ into \eqref{3.5-3} and calculate

\[
\genfrac{[}{]}{0pt}{}{2}{1}=\sum^2_{i=2}T^{(2)}_{1,2-i}\binom{4-i}{2}=T^{(2)}_{1,0}\binom{2}{2}=1
\]
and $s(2,1)=(-1)^{2-1}\genfrac{[}{]}{0pt}{}{2}{1}=-1$.

Similarly, substituting $(n,k)=(3,2)$ and $(n,k)=(4,2)$ into \eqref{3.5-3}, respectively, yields 

\begin{align*}
&\genfrac{[}{]}{0pt}{}{3}{1}=\sum^3_{i=3}T^{(2)}_{2,4-i}\binom{7-i}{4}=T^{(2)}_{2,1}\binom{4}{4}=2\\
&\genfrac{[}{]}{0pt}{}{4}{2}=\sum^4_{i=3}T^{(2)}_{2,4-i}\binom{8-i}{4}=T^{(2)}_{2,1}\binom{5}{4}+T^{(2)}_{2,0}\binom{4}{4}=2\cdot 5+1\cdot 1=11.
\end{align*}
Thus $s(3,1)=2$ and $s(4,2)=11$.

We may also use \eqref{3.5-4} to calculate the second-order Eulerian numbers $T^{(2)}_{u,v}$ by substituting $(k,i)=(u, 2u-v)$. For instance, to find $T^{(2)}_{2,1}$ by using \eqref{3.5-4}, we substitute $k=2$ and $i=3$ into \eqref{3.5-4} and evaluate 

\begin{align*}
&T^{(2)}_{2,1}=\sum^3_{n=0}(-1)^{3-n}\genfrac{[}{]}{0pt}{}{n}{n-2}\binom{5}{3-n}\\
=&(-1)^0\genfrac{[}{]}{0pt}{}{3}{1}\binom{5}{0}=2.
\end{align*}

Similarly,

\begin{align*}
&T^{(2)}_{3,1}=\sum^5_{n=0}(-1)^{5-n}\genfrac{[}{]}{0pt}{}{n}{n-3}\binom{7}{5-n}\\
=&-\genfrac{[}{]}{0pt}{}{4}{1}\binom{7}{1}+\genfrac{[}{]}{0pt}{}{5}{2}\binom{7}{0}\\
=&-6\cdot 7+50\cdot 1=8.
\end{align*}
And 

\begin{align*}
&T^{(2)}_{4,2}=\sum^6_{n=0}(-1)^{6-n}\genfrac{[}{]}{0pt}{}{n}{n-4}\binom{9}{6-n}\\
=&-\genfrac{[}{]}{0pt}{}{5}{1}\binom{9}{1}+\genfrac{[}{]}{0pt}{}{6}{2}\binom{9}{0}=-24\cdot 9+274\cdot 1=58.
\end{align*}
\end{example}

\noindent{\bf Remark} From Theorems \ref{thm:3.4}-\ref{thm:3.6}, we may establish the following relationship between unsigned Stirling numbers of the first kind and the Stirling numbers of the second kind:

\begin{align*}
&S(n+\ell,\ell)\\
=&\sum^{\ell-1}_{k=1}\sum^{2n-k}_{j=0}(-1)^{2n-k-j}\binom{2n+1}{2n-k-j}\binom{2n-k+\ell-1}{2n}\genfrac{[}{]}{0pt}{}{j}{j-n},\\
&\genfrac{[}{]}{0pt}{}{n}{n-k}\\
=&\sum^n_{i=k+1}\sum^{2k-i+1}_{\ell =1}(-1)^{2k-\ell-i+1}\binom{2k+1}{2k-\ell-i+1}\binom{2k+n-i}{2k}S(k+\ell, \ell).
\end{align*}

We now define a function in terms of the $m$th-order Eulerian numbers.

\begin{definition}\label{def:3.6}
Let $T^{(m)}_{n,k}$ be the $m$th-order Eulerian numbers defined by Definition \ref{def:2.1}. Define function $\phi_{m;n}:{\bR}\to {\bR}$ by 

\be\label{3.6}
\phi_{m;n}(x)=\sum^{n-1}_{k=0} T^{(m)}_{n,k}\binom{x+k}{mn}=\sum^{n-1}_{k=0} T^{(m)}_{n,k}\frac{(x+k)_{mn}}{mn},
\ee
where $(x)_\ell$ is the falling factorial defined as the polynomial, $(x)_\ell=x(x-1)\cdots (x-\ell+1)$.

Particularly, 

\begin{align}
&\phi_{1;n}(x)=x^n,\label{3.6-2}\\
&\phi_{2;n}(x)=\left[ \begin{array} {c} x\\x-n\end{array}\right].\label{3.6-3}
\end{align}
\end{definition}

Equations \eqref{3.6-2} and \eqref{3.6-3} are known as Worpitzky's identities, where \eqref{3.6-2} can be seen in Worpitzky \cite{Wor}. Two different proofs of \eqref{3.6-2} can be found in \cite{Com} and \cite{GKP}, while \eqref{3.6-3} is proved by induction in \cite{GKP}. 

For $m=3$, we have 

\begin{align*}
&\phi_{3;1}(x)=\binom{x}{3},\\
&\phi_{3;2}(x)=\binom{x}{6}+3\binom{x+1}{6},\\
&\phi_{3;3}(x)=\binom{x}{9}+12\binom{x+1}{9}+15\binom{x+2}{9},
\end{align*}
etc. However, the closed forms of $\phi_{m;n}(x)$, $m\geq 3$, are unknown.

\section{Values of $T^{(m)}_{n,k}$}

We now establish the formula for the values of $m$th-order Eulerian numbers by using a Fa\`a di Bruno's type expression, which presents a relationship between partitions and permutations. 

\begin{theorem}\label{thm:2.5}
Let $T^{(m)}_{n,k}$ be the $m$th-order Eulerian numbers defined in Definition \ref{def:2.1}. Then

\begin{align}\label{2.13}
&T^{(m)}_{n,k}=\sum_{\substack {t_1+t_2+\cdots +t_{k+1}=n-k-1\\t_1, \ldots, t_{k+1}\geq 0}}1^{t_1}2^{t_2}\ldots (k+1)^{t_{k+1}}(mt_1+m)\nonumber\\
&(m(t_1+t_2)+2(m-1)+1)\ldots (m(t_1+\cdots+t_{k})+k(m-1)+1).
\end{align}

Particularly, 

\begin{align}\label{2.13-2}
&T^{(1)}_{n,k}=\langle \begin{array}{cc}n\\k\end{array}\rangle=\sum_{\substack {t_1+t_2+\cdots +t_{k+1}=n-k-1\\t_1, \ldots, t_{k+1}\geq 0}}1^{t_1}2^{t_2}\ldots (k+1)^{t_{k+1}}\nonumber\\
&\qquad \qquad \qquad (t_1+1)(t_1+t_2+1)\ldots (t_1+\cdots+t_{k}+1).
\end{align}
\end{theorem}

\begin{proof} We prove \eqref{2.13} on $n$ by using induction. Since $k\leq n$, for $n=1$ \eqref{2.13} holds because the both sides of the equation are equal to $1$.

Assume \eqref{2.13} holds for $n$ with $0\leq k\leq n-1$, we consider $T^{(m)}_{n+1,k}$ with $0\leq k\leq n$ by applying 

\[
T^{(m)}_{n+1, k}=(k+1)T^{(m)}_{n,k}+(mn-k+1)T^{(m)}_{n,k-1}.
\]
Hence,

\begin{align*}
&T^{(m)}_{n+1, k}= (k+1) \sum_{\substack {t_1+t_2+\cdots +t_{k+1}=n-k-1\\t_1, \ldots, t_{k+1}\geq 0}}1^{t_1}2^{t_2}\ldots (k+1)^{t_{k+1}}(mt_1+m)\\
&(m(t_1+t_2)+2(m-1)+1)\ldots (m(t_1+\cdots+t_{k})+k(m-1)+1)\\
&+(mn-k+1)\sum_{\substack {t_1+t_2+\cdots +t_{k}=n-k\\t_1, \ldots, t_{k}\geq 0}}1^{t_1}2^{t_2}\ldots k^{t_{k}}(mt_1+m)\\
&(m(t_1+t_2)+2(m-1)+1)\ldots (m(t_1+\cdots+t_{k-1})+(k-1)(m-1)+1)\\
=&\sum_{\substack {t_1+t_2+\cdots +t_{k+1}=n-k-1\\t_1, \ldots, t_{k+1}\geq 0}}1^{t_1}2^{t_2}\ldots (k+1)^{t_{k+1}+1}(mt_1+m)\\
&(m(t_1+t_2)+2(m-1)+1)\ldots (m(t_1+\cdots+t_{k})+k(m-1)+1)\\
&+\sum_{\substack {t_1+t_2+\cdots +t_{k}=n-k\\t_1, \ldots, t_{k}\geq 0}}1^{t_1}2^{t_2}\ldots k^{t_{k}}(mt_1+m)
(m(t_1+t_2)+2(m-1)+1)\\
&\ldots (m(t_1+\cdots+t_{k-1})+(k-1)(m-1)+1)(m(n-k)+k(m-1)+1)\\
=&\sum_{\substack {s_1+s_2+\cdots +s_{k+1}=n-k\\s_1, \ldots, s_{k}\geq 0, s_{k+1}\geq 1}}1^{s_1}2^{s_2}\ldots (k+1)^{s_{k+1}}(ms_1+m)\\
&(m(s_1+s_2)+2(m-1)+1)\ldots (m(s_1+\cdots+s_{k})+k(m-1)+1)\\
&+\sum_{\substack {t_1+t_2+\cdots +t_{k}+t_{k+1}=n-k\\t_1, \ldots, t_{k}\geq 0, t_{k+1}=0}}1^{t_1}2^{t_2}\ldots k^{t_{k}}(k+1)^{t_{k+1}}(mt_1+m)\\
&(m(t_1+t_2)+2(m-1)+1)\ldots (m(t_1+\cdots+t_{k-1})+(k-1)(m-1)+1)\\
&(m(t_1+\cdots+t_{k})+k(m-1)+1)\\
=&\sum_{\substack {t_1+t_2+\cdots +t_{k}+t_{k+1}=n-k\\t_1, \ldots, t_{k+1}\geq 0}}1^{t_1}2^{t_2}\ldots k^{t_{k}}(k+1)^{t_{k+1}}(mt_1+m)\\
&(m(t_1+t_2)+2(m-1)+1)\ldots 
(m(t_1+\cdots+t_{k})+k(m-1)+1).
\end{align*}
Therefore \eqref{2.13} holds for $n+1$ with $0\leq k\leq n$, which completes the proof of the theorem.
\end{proof}

\begin{example}
As examples of the formula \eqref{2.13}, the value of $T^{(3)}_{5,2}$ is 

\begin{align*}
&T^{(3)}_{5,2}=\sum_{\substack{t_1+t_2+t_3=2\\t_1,t_2,t_3\geq 0}}1^{t_1}2^{t_2}3^{t_3}(3t_1+3)(3(t_1+t_2)+5)\\
=&1^22^03^0(6+3)(6+5)+1^02^23^0(0+3)(6+5)+1^02^03^2(0+3)(0+5)\\
&+1^12^13^0(3+3)(6+5)+1^12^03^1(3+3)(3+5)+1^02^13^1(0+3)(3+5)\\
=&99+132+135+132+144+144=786.
\end{align*}

And the value of $T^{(4)}_{5,2}$ is 

\begin{align*}
&T^{(4)}_{5,2}=\sum_{\substack{t_1+t_2+t_3=2\\t_1,t_2,t_3\geq 0}}1^{t_1}2^{t_2}3^{t_3}(3t_1+4)(3(t_1+t_2)+7)\\
=&1^22^03^0(8+4)(8+7)+1^02^23^0(0+4)(8+7)+1^02^03^2(0+4)(0+7)\\
&+1^12^13^0(4+4)(8+7)+1^12^03^1(4+4)(4+7)+1^02^13^1(0+4)(4+7)\\
=&1440.
\end{align*}

\end{example}
Substituting $k=n-1$ and $k=1$ into \eqref{2.13}, we may find the values of $T^{(m)}_{n,n-1}$ and $T^{(m)}_{n,1}$ as follows. 

\begin{corollary}\label{cor:2.5}
Let $T^{(m)}_{n,k}$ be the $m$th-order Eulerian numbers defined in Definition \ref{def:2.1}. Then

\begin{align}\label{2.14}
&T^{(m)}_{n,n-1}=m (2(m-1)+1)\ldots((n-1)(m-1)+1),\\
&T^{(m)}_{n,1}=m(2^n-n-1).
\label{2.15}
\end{align}
\end{corollary}

\begin{theorem}\label{thm:2.6}
Let $T^{(m)}_{n,k}$ be the $m$th-order Eulerian numbers defined in Definition \ref{def:2.1}. Then

\be\label{2.16}
\sum^{n-1}_{k=0} T^{(m-1)}_{n,k}=T^{(m)}_{n,n-1}.
\ee
\end{theorem}

\begin{proof}
Equation \eqref{2.11} is derived from the comparison of \eqref{2.12} and \eqref{2.14}.
\end{proof}

\noindent{\bf Remark 4.1}
There is a combinatorial proof of Equation \eqref{2.16}. For instance, we know that $T^{(2)}_{2,0}=1$ counts the permutation in $Q_4$ with no descent, i.e., $1122$, and $T^{(2)}_{2,1}=2$ is the number of permutations with single descent: $2211$ and $2112$. To prove 

\[
\sum^1_{k=0}T^{(2)}_{2,k}=T^{(2)}_{2,0}+T^{(2)}_{2,1}=3=T^{(3)}_{2,1},
\]
we may insert $21$ at the beginning of $1122$ and each descent of $2211$ and $2112$ to obtain three permutations in $Q_6$ with single descent: $211122$, $222111$, and $221112$, which are exactly three permutations counted by $T^{(3)}_{2,1}$.

\end{document}